\documentclass[graybox]{svmult}
\usepackage{newtxtext, newtxmath,graphicx,makeidx,multicol,footmisc} 
\usepackage[utf8]{inputenc}
\usepackage{graphicx}
\usepackage{color,epsfig}

\newtheorem{observation}{Observation}
\def\qed{\ifhmode\unskip\nobreak\fi\quad\ifmmode\Box\else$\Box$\fi}

\long\def\comment#1{}

\newcommand{\RR}{{\mathbb R}}

\def\cro{{\mbox {\sc cr}}}
\def\ocr{{\mbox {\sc ocr}}}
\def\pcr{{\mbox {\sc pcr}}}

\title{Generalizations of the Crossing Lemma}
\author{G\'eza T\'oth
\thanks{Supported by National Research, Development and Innovation Office, NKFIH,
K-131529 and ERC Advanced Grant “GeoScape" 882971.}\\
\small HUN-REN Alfr\'ed R\'enyi Institute of Mathematics,  Budapest, Hungary\\
\small Budapest University of Technology and Economics, SZIT, Budapest, Hungary\\
\small \texttt{geza@renyi.hu}}

\begin{document}
\maketitle


\abstract{
The crossing number of a graph is the minimum number of crossings over all of its drawings on the plane. The Crossing Lemma, proved more than 40 years ago, is  a tight lower bound on the crossing number of a graph in terms of the number of vertices and edges. It is definitely the most important inequality on crossing numbers.
We review some generalizations and  applications of the Crossing Lemma.}


\section{Introduction}

The {\em crossing number} $\cro(G)$ of $G$ is the minimum number of edge
crossings over all drawings of $G$ on the plane.
It has a fascinating history and lots of applications in theory and practice.
It started with Turán's Brick Factory problem in  1944, also known
as Zarankiewicz's problem \cite{G69}. 
Determining or even estimating the crossing number is a notoriously difficult problem \cite{GJ83, S18},
and it has a huge number of practical and theoretical applications \cite{S97, DLM19} in 
VLSI design, geographic information systems, graph drawing, computer graphics,
and many other fields.
Even the crossing numbers of complete graphs
are not known, despite the huge number of results, attempts, improvements
\cite{BLS19, HH62, ACF12, ADF20}.
%
The Crossing Lemma of Ajtai, Chvátal, Newborn, Szemerédi and independently
Leighton, 40 years ago was a real breakthrough.
It is a general lower bound on the crossing number, 
it states that any simple graph with $n$ vertices and $e\ge 4n$ edges has crossing number
at least $ce^3/n^2$ for some $c>0$.

Székely \cite{S97} and then others observed that crossing numbers, in particular the Crossing Lemma, 
also has important
applications in the theory of incidences \cite{ST01} , additive number theory \cite{E97}, 
in the analysis of geometric algorithms \cite{D98} and many other fields.

In this paper we survey some results related to the Crossing Lemma and its generalizations.
It Section 2 we review the history and improvements of the Crossing Lemma.
It is easy to see that 
it does not hold for multigraphs in general, not even in any weaker form.   
In Section 3
we investigate its generalizations for multigraphs under different conditions on the drawing. 
In Section 4 we try to figure out, how far can we go with such a generalization, that is, what could be the weakest condition which implies some Crossing Lemma type statement. 
Finally, in Section 5 we introduce some other versions of the crossing number, the pair-crossing number and the odd-crossing number, their relationships, and generalizations of the Crossing Lemma for these crossing numbers.

\section{The Crossing Lemma}\label{sec:crlemma}

In a {\em drawing} of a graph (or multigraph) $G$, vertices are represented by points and edges
are represented by curves connecting the corresponding points. 
If it does not lead to confusion, the points (curves), representing the vertices (edges) are also called vertices (edges).
We assume that in a drawing an edge cannot contain a vertex in its interior, if two edges have a common point, then it is a common endpoint or a proper crossing, and three edges do not cross at the same point.
For any graph $G$, $n(G)$ and $e(G)$ denote its number of vertices and edges, respectively.
A graph (resp. multigraph), together with its drawing, is called a {\em topological} 
graph (resp. multigraph).

The most important, fundamental result on crossing numbers is the Crossing Lemma,
first proved by Ajtai, Chv\'atal, Newborn, Szemer\'edi, and
Leighton, $40$ years ago
\cite{ACNS82, L83}.

\smallskip

\begin{theorem}[Crossing Lemma]\label{crossinglemma}
For any simple graph $G$, if
$e(G)\ge 4n(G)$ 
then $$\cro(G)\ge \frac{1}{64}\frac{e(G)^3}{n(G)^2}.$$ 
\end{theorem}

\smallskip

The original proofs were quite complicated and gave a much worse constant in the bound. 
The following beautiful proof is attributed
to Chazelle, Sharir, and Welzl and appeared in 
"Proofs from THE BOOK" \cite{AZ99}.
For the proof we need the following easy but very important statement (see e.g. \cite{PT00}).
\begin{observation}\label{swap}
For any simple graph $G$, in a drawing with exactly $\cro(G)$ crossings, 
(i) no edge crosses itself, (ii) 
two edges have at most one common point. 
\end{observation}

\begin{proof}
(i) Suppose that an edge $\alpha$ crosses itself at point $p$. 
Reverse the orientation of the part of $\alpha$ between the two occurrences of $p$. Then $p$ becomes a self-touching point. And now we can get rid of the touching by a slight perturbation at $p$. 
This way e get a drawing 
with fewer crossings, which is a contradiction.
See Figure \ref{swapfig}.

(ii) Suppose that two edges $\alpha$ and $\beta$ have two common points, $p$ and $q$. 
One of them can be a crossing or a common endpoint, the other one is a crossing.
Then swap the parts of the edges between $p$ and $q$. Again, we get a drawing 
with fewer crossings, which is a contradiction. See Figure \ref{swapfig} \hfill 
\end{proof}


\begin{figure}[!ht]
\begin{center}
\scalebox{0.5}{\includegraphics{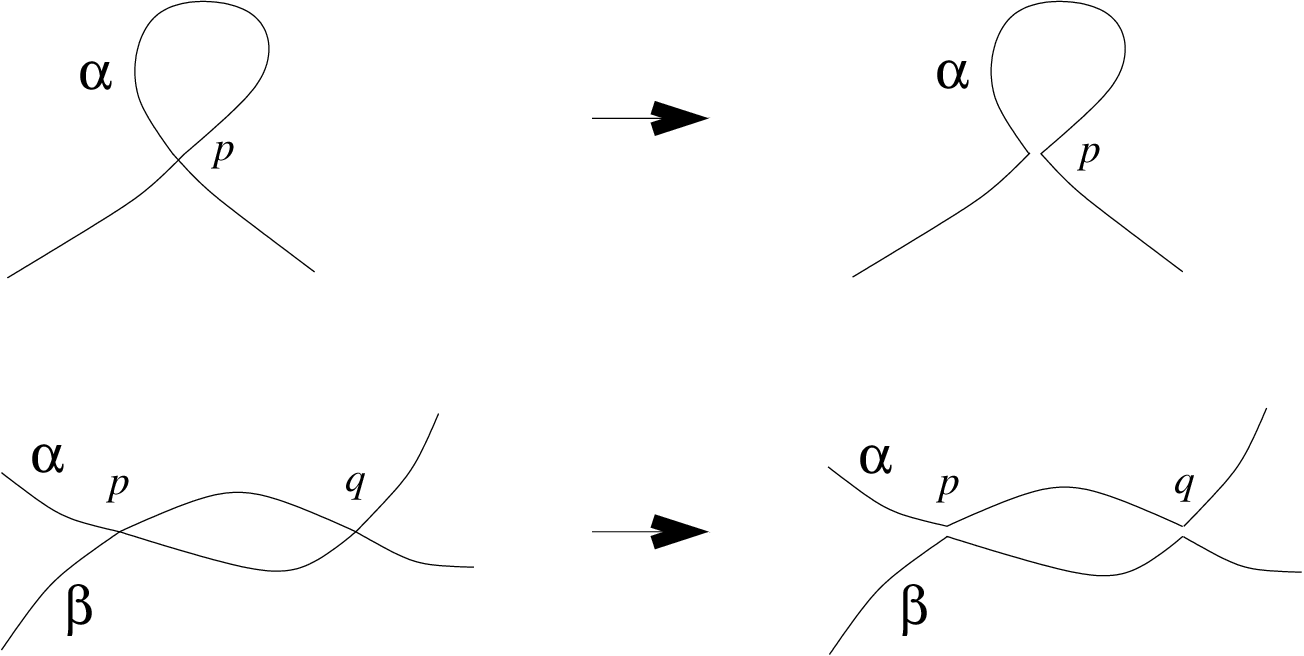}}
\caption{Reducing the number of crossings.}\label{swapfig}
\end{center}
\end{figure}


\begin{proof}[Proof of the Crossing Lemma] 
The proof goes in two steps. 

1. First we show that for any graph $G$ 
\begin{equation}\label{e-3n}
\cro(G)\ge e(G)-3n(G).
\end{equation}
This can be proved by induction on $e(G)$. If $e(G)\le 3n(G)$, then the statement holds trivially.
Suppose that $e(G)>3n(G)$ and we have already proved the statement for smaller values of $e(G)$. 
It is well-known consequence of Euler's formula that a planar graph of $n$ vertices has 
at most $3n-6$ edges. 
Therefore,  $\cro(G)>0$. Remove an edge $\alpha$ from $G$ which contains a crossing. 
Then $\cro(G)\ge \cro(G-\alpha)+1$ and by the induction hypothesis, 
$\cro(G-\alpha)\ge e(G)-1-3n(G)$. Combining them, the statement follows.

2. 
Consider a drawing of $G$ with $\cro(G)$ crossings. Take a random induced subgraph
$G'$ of $G$ by selecting each vertex independently with probability $p$ (whose value is determined later). 
We obtain a random subset of vertices $U\subset V(G)$ and let $G'$ be the
subgraph induced by $U$. 
Draw the edges of 
$G'$ as they were drawn in the original drawing.
The number of vertices and edges of $G'$, $n(G')$ and $e(G')$ are random variables, and
$E[n(G')]=pn(G)$, 
$E[e(G')]=p^2e(G)$ since each edge of $G$ survives with probability $p^2$.
Let $X(G')$ be the number of crossings of $G'$ in the obtained drawing. 
Then clearly $X(G')\ge\cro(G')$. By Observation \ref{swap}, only independent edges can cross in the original optimal drawing of $G$, consequently, also in $G'$.
A crossing is present in the drawing of $G'$ if all four corresponding endpoints are selected and  the probability of this is $p^4$. (A crossing of two edges with a common endpoint would be present with probability $p^3$.) 
Therefore, $E[X(G')]=p^4\cro(G)$. By  \ref{e-3n}, $\cro(G')\ge e(G')-3n(G')$. Taking expectations, we get
$$p^4\cro(G)=E[X(G')]\ge E[\cro(G')]\ge E[e(G')]-3E[n(G')]=p^2e(G)-3pn(G)$$
so $\cro(G)\ge e(G)/p^2-3n(G)/p^3$.
Set $p=4n(G)/e(G)$,  by assumption $p\le 1$. 
Substitute in the last formula and we get
$$\cro(G)\ge\frac{1}{64}e(G)^3/n(G)^2.$$ 
\end{proof}

This bound cannot be improved apart
from the value of the constant.
That is, there is a constant $c$ with the following property. 
For any $n$, $e$, $e\le \binom{n}{2}$, 
there is a graph $G$ with $n(G)=n$, $e(G)=e$ and $\cro(G)\le ce^3/n^2$.
The easiest example is --- very roughly --- the following. 
Take the disjoint union of $n^2/2e$ complete graphs, each of $2e/n$ vertices. 
This graph has approximately $n$ vertices, $e$ edges. Each component can be drawn such that any two edges cross at most once, and components can be separated from each other. 
Hence  $\cro(G)\le (n^2/2e)((2e/n)^2/2)^2/2=e^3/n^2$.

There is a better example by Pach and T\'oth \cite{PT97}.  Suppose that $n\gg e\gg n^2$, that is, 
$\lim_{n\rightarrow\infty} e/n=\infty$ and $\lim_{n\rightarrow\infty}  n^2/e=\infty$.
We define a graph $G$ of $n$ vertices and $e$ edges, together with its drawing.
Let 
$V(G)$ be a set of $n$ points arranged in
a slightly perturbed unit square grid
of size $\sqrt{n}\times \sqrt{n}$, so that the points
are in general position.
Let $r=\sqrt{2e/\pi n}$, so that $r^2\pi=2e/n$.
Connect two points by a straight-line segment if and only if
their distance is at most $r$. Then,  $n(G)=n$, the degrees of almost all vertices are close to $r^2\pi$, so 
$e(G)\approx nr^2\pi/2=e$. One can approximate the number of crossings, which is an upper bound for the crossing number, 
by an integral. 
It turns out that 
\begin{equation}\label{alsokorlat}
\cro(G)\le \left(8/9\pi^2+o(1)\right)e^3/n^2<\left( 0.09006+o(1)\right)e^3/n^2.
\end{equation}

In fact, the calculation in \cite{PT97} includes a quadruple integral and there was an error in it, which was corrected in
\cite{PRTT06}.  Recently, Czabarka et al. \cite{CSSW20} found essentially the same bound with a much simpler and more general construction.
Take $n$ random points independently on the sphere. Connect a pair of points by the geodesic (the shortest connecting curve, which is an arc on a great circle) 
if it has length at most $r$. If we set the value of $r$ properly, the expected number of edges is $e$ and the expected number of crossings is $(8/9\pi^2+o(1))e^3/n^2$. It is not hard to argue that there is an instance
where both the number of edges and crossings are close to their expected values.

From the other direction there were several improvements. In general, for a better constant we have to pay with 
a stronger condition in place of $e\ge 4n$.
First, in the proof above we can set $p=4.5n/e$ but then we have to assume that $e\ge 4.5n$.
We obtain the following statement: 

$$ \mbox{if }\  e(G)\ge 4.5n(G) \ 
\mbox{then }\  \cro(G)\ge \frac{1}{60.75}\frac{e(G)^3}{n(G)^2}.$$ 

The next improvement was in \cite{PT97}. 
The main idea is that equation (\ref{e-3n}) can be replaced by a stronger one. In equation \ref{e-3n} we only use that once a graph has more than 
$3n-6$ edges, then there is an edge with a crossing on it. 
But if the number of edges is a little bit more, then there will be an edge with {\em two} or even more crossings.

For any $k\ge 0$, a graph is called  
{\em $k$-planar} if it has a drawing
where every edge contains at most $k$ crossings.
Let $e_k(n)$ denote the maximum number of edges of a $k$-planar graph of $n$
vertices.   Clearly, $0$-planar graphs are exactly the planar graphs, therefore, 
$m_0(n)=3n-6$.
It was shown in \cite{PT97} that 
$m_1(n)=4n-8$ for $n\ge 12$ and 
$m_2(n)\le 5n-10$ which is tight for infinitely
many values of $n$, 
$m_3(n)\le 6n-12$ and  $m_4(n)\le 7n-14$.

Therefore, for any graph $G$ 
$$\cro(G)\ge (e(G)-3n(G))+(e(G)-4n(G))+(e(G)-5n(G))$$
$$+(e(G)-6n(G))+
(e(G)-7n(G))=5e(G)-25n(G).$$
By Observation \ref{swap}, all crossings are between independent edges. 
Using this inequality in the probabilistic argument in the proof of the Crossing Lemma we get the following.

\begin{theorem}[\cite{PRTT06}]\label{33.75}
For any simple graph $G$, if
$e(G)\ge 7.5n(G)$ 
then $\cro(G)\ge \frac{1}{33.75}\frac{e(G)^3}{n(G)^2}$. 
\end{theorem}

Ten years later, it was shown by Pach, Radoi\v ci\'c, Tardos and T\'oth \cite{PRTT06} that 
$m_3(n)\le 5.5n-11$, which is tight apart from an additive constant. Using this bound, we already get an 
improvement, but they added some further structural observations and obtained that 
$$\cro(G)\ge 7e(G)/3-25n(G)/3.$$
This implies the following result, 
by the usual probabilistic argument.

\begin{theorem}[\cite{A19}]\label{31.1}
For any simple graph $G$, if
$e(G)\ge 17.16n(G)$ 
then $\cro(G)\ge \frac{1}{31.1}\frac{e(G)^3}{n(G)^2}$. 
\end{theorem}

The next improvement was by Ackerman \cite{A19}. He proved that 
$m_4(n)\le 6n-12$,
which is also tight up to an additive constant.
He obtained the following  bound.

\begin{theorem}[\cite{A19}]\label{29}
For any simple graph $G$, if
$e(G)\ge 6.95n(G)$ 
then $\cro(G)\ge \frac{1}{29}\frac{e(G)^3}{n(G)^2}$. 
\end{theorem}

Finally, the best known bound is a recent result of 
B\"ungener and Kaufmann \cite{BK24}. They use the same approach and some additional structural 
results on near maximal $2$-planar and $3$-planar graphs.
They obtained the following linear bound.
$$\cro(G)\ge 5e(G)-203n(G)/9$$
which implies 

\begin{theorem}[\cite{BK24}]\label{27.48}
For any simple graph $G$, if
$e(G)\ge 6.77n(G)$ 
then $\cro(G)\ge \frac{1}{27.48}\frac{e(G)^3}{n(G)^2}$. 
\end{theorem}

Note that in each case we could apply Obseration \ref{swap} to conclude that 
all crossings are between independent edges. This is necessary for the probabilistic argument.

For larger $k$, $m_k(n)=\Theta(\sqrt{k}n)$, the upper bound is also a consequence of the Crossing Lemma \cite{HT20}.

So we do not know the optimal constant in the Crossing Lemma, but 
at least we know that
it exists in the following sense.
Let $\kappa(n,e)$ denote the minimum crossing number of a
graph $G$ with $n$ vertices and at least $e$ edges. That is,
$$\kappa(n,e)=\min_{\begin{array}{cc}
n(G)=n \\ e(G)\ge e \end{array}}\cro(G).$$


Verifying a conjecture of 
Erd\H os and
Guy \cite{EG73}, 
Pach, Spencer and T\'oth proved the following.

\begin{theorem}[\cite{PST00}]\label{midrange} If 
$n\ll e\ll n^2$, then
$$\lim_{n\rightarrow\infty}\kappa(n,e)\frac{n^2}{e^3}=c>0.$$
\end{theorem}

It is not known whether the condition $n\ll e\ll n^2$ can be replaced by the weaker condition
$C_1n\le e\le C_2n^2$. If the answer is yes, then an easy calculation shows that $C_1>4$ and $C_2<1/2$.
We call the constant $c>0$ in Theorem \ref{midrange} the {\em midrange
crossing constant}.
As we have seen, the value of $c$ is still unknown but we know that  
$1/27.48\le c\le 8/9\pi^2$.
We have $1/27.48\approx 0.036$ and $8/9\pi^2\approx 0.09$, 
so the upper and lower bounds are within a factor of $2.5$.

It was shown by Czabarka et al. \cite{CSSW20} that there is a midrange crossing constant $c_b$, defined analogously,
for {\em bipartite graphs} as well. The best known bounds are 
$0.055\le c_b\le 0.18$ \cite{ABK17}. They also proved the existence of a midrange crossing constant for many other graph classes.


\section{Multigraphs}

The Crossing Lemma does not hold 
for multigraphs in general. To see this, 
just take two vertices and arbitrarily many parallel edges connecting them, with no crossings, or just take one vertex and many loops.
So it is important to understand, under what conditions on the drawings
does the statement of the Crossing
Lemma, or a similar statement holds for multigraphs. 
For simplicity, we assume that there are no loops, but there could be parallel edges in the multigraphs considered in this section.

The first such attempt to generalize the Crossing Lemma for multigraphs 
was to limit the multiplicity of parallel edges \cite{PT97, S97, M02}. 
In this case we have an easy, weaker, but tight bound. 
But this condition does not really give an insight.

\begin{theorem}[Crossing Lemma for multigraphs with bounded edge-multiplicity]\label{crlemma-multiplicity}
For any multigraph $G$, if any two vertices are connected by at most $m$ edges and 
$e(G)\ge 4mn(G)$ 
then $$\cro(G)\ge \frac{1}{64}\frac{e(G)^3}{mn(G)^2}.$$ 
\end{theorem}

\begin{proof}
%

Consider a drawing of $G$ with $\cro(G)$ crossings. Take a random induced subgraph
$G''$ by the following random process. Let $p=4mn/e$ which is at most $1$, by the assumption.
First 
select each vertex independently with probability $p$. 
We obtain a random subset of vertices $U\subset V(G)$ and let $G'$ be the
subgraph induced by $U$.
Now for each pair of adjacent vertices, $u, v\in U$, de the following. 
Keep  each of the edges between $u$ and $v$ with probability 
$1/m$, and delete the others, in such a way that these events are mutually exclusive. 
The resulting random graph is $G''$.
The edges of 
$G''$ are drawn as they were drawn in the original drawing.
The number of vertices and edges of $G''$, $n(G'')$ and $e(G'')$ are random variables, and
$E[n(G')]=pn(G)$, 
$E[e(G')]=p^2e(G)/m$ since each edge of $G$ survives with probability $p^2/m$.
For the  number of crossings $X(G'')$ of $G''$ in the obtained drawing we have  
$X(G')\ge\cro(G')$. 
The probability that a crossing is present in the drawing of $G'$ is $p^4/m^2$, so 
$E[X(G'')]=p^4/m^2\cro(G)$. Since $G''$ is a simple graph, by  (\ref{e-3n}), 
$\cro(G'')\ge e(G'')-3n(G'')$, so 
$$p^4/m^2\cro(G)=E[X(G'')]\ge E[\cro(G'')]$$
$$\ge E[e(G'')]-3E[n(G'')]=p^2/me(G)-3pn(G)$$
and we get that 
$$\cro(G)\ge\frac{1}{64}\frac{e(G)^3}{n(G)^2m}.$$ 
\end{proof}

For $m=1$, Theorem \ref{crlemma-multiplicity} 
gives the Crossing Lemma, but as $m$ increases, the bound is
getting weaker. Still, it is also 
tight up to
a constant factor. 
Moreover, it was shown by \'Agoston and P\'alv\"olgyi \cite{AP22}
that Theorem 
\ref{crlemma-multiplicity}
holds with the same constant as the the original Crossing Lemma, in the following sense. 
Let $\kappa(n,e, m)$ denote the minimum crossing number of a
multigraph $G$ with $n$ vertices, at least $e$ edges, and edge multiplicity at most $m$.
In particular, $\kappa(n,e, 1)=\kappa(n,e)$. 

\begin{theorem}[\cite{AP22}]\label{agostonpalvolgyi} 
For any $m\ge 1$, 
$$m^2\kappa\left(n, \lfloor e/m\rfloor \right)\le \kappa(n, e, m)\le m^2\kappa\left(n, \lceil e/m\rceil \right).$$
\end{theorem}

\begin{proof}
For the upper bound on $\kappa(n, e, m)$, take a simple graph with $n$ vertices, 
$\lceil e/m\rceil$
edges drawn with $\kappa\left(n, \lceil e/m\rceil \right)$ crossings.
Now replace each edge by at most $m$ parallel edges, drawn very close to the original edge. 
This way we can get a multigraph of $e$ edges and edge multiplicity 
at most $m$. Each original crossing corresponds to at most $m^2$ new crossings, so we have
$\kappa(n, e, m)\le m^2\kappa\left(n, \lceil e/m\rceil \right)$.

Now we prove the lower bound on  $\kappa(n, e, m)$. 
Two parallel edges are called {\em homotopic} if they are homotopic in the vertex-punctured plane. 
That is, the closed curve formed by their union is contractible to a point. 
Take a multigraph $G$ with $n$ vertices, $e$ edges, edge mutiplicity at most $m$, drawn with 
$\kappa(n, e, m)$ crossings. Let $u, v$ be two vertices and take all $uv$ edges. 
Let $\alpha$ be one with the minimum number of crossings. Redraw all other parallel edges 
very close to $\alpha$. We did not increase the number of crossings, so we still have 
$\kappa(n, e, m)$ crossings. Do it with all pairs of vertices. Now we have a drawing with 
$\kappa(n, e, m)$ crossings such that parallel edges are homotopic.

An edge is called {\em full} if it has multiplicity $m$.
Suppose that $\alpha$ and $\beta$ are not parallel, $\alpha$ is 
not full, and it has at most as many crossings 
as $\beta$. Remove $\beta$ and add an edge parallel to $\alpha$, drawn very close to it.
We obtain another multigraph with the same number of edges, edge multiplicity at most $m$, and we did not increase the number of crossings. Therefore, we still have $\kappa(n, e, m)$ crossings. By a repeated application of this step, we can obtain a multigraph $G'$ with $n$ vertices, $e$ edges, edge mutiplicity at most $m$, drawn with 
$\kappa(n, e, m)$ crossings, such that the number of non-full edges is less than $m$, and they are all parallel.

Now delete all non-full edges and keep only one of each parallel $m$-tuples of the the remaining edges.
Let $G_1$ be the obtained simple graph. It has $n$ vertices, $\lfloor e/m\rfloor$
edges and in its drawing each crossing corresponds to $m^2$ original crossings. Therefore, 
$m^2\kappa\left(n, \lfloor e/m\rfloor \right)\le \kappa(n, e, m)$. \hfill
\end{proof}


It was suggested by Bekos, Kaufmann, and Raftopoulou
\cite{K16}
that the dependence on the
multiplicity might be eliminated if we restrict our attention to special
classes of drawings. In particular, it is natural to assume that each "lens" enclosed by two 
parallel edges contains a vertex.

\medskip

\begin{definition}\label{3def}
A {\em drawing} of a multigraph $G$ in the plane is called


-- {\bf separated} if any pair of parallel edges 
form a simple closed curve and there is at least one vertex in its interior and exterior,

-- {\bf single-crossing} if any two independent edges cross at most once,

-- {\bf locally starlike} if adjacent edges do not cross.

\end{definition}

It turned out that these conditions are sufficient to generalize the Crossing Lemma in some form.
The first such result is by Pach and T\'oth \cite{PT20}.

\begin{theorem}[\cite{PT20}]\label{sep+sing+lsl} 
Suppose that $M$ is a separated, single-crossing, locally starlike topological multigraph and
$e(M)\ge 4n(M)$.
Then $\cro(M)\ge c\frac{e(G)^3}{n(G)^2}$
for some constant $c>0$. 
\end{theorem}

Later it turned out that the "single-crossing" condition can be dropped. 

\begin{theorem}[\cite{KPTU21}]\label{sep+lsl} 
Suppose that $M$ is a separated, locally starlike topological multigraph and
$e(M)\ge 4n(M)$.
Then $$\cro(M)\ge c\frac{e(G)^3}{n(G)^2}$$
for some constant $c>0$. 
\end{theorem}

Both of these results are asymptotically thight. If we have the "separated" and "single-crossing" conditions, 
then we have a slightly weaker result.

\begin{theorem}[\cite{FPS21}]\label{sep+sing} 
Suppose that $M$ is a separated, single-crossing topological multigraph and
$e(M)\ge 4n(M)$.
Then $$\cro(M)\ge c\frac{e(G)^3}{n(G)^2\log(e/n)}$$
for some constant $c>0$. 
\end{theorem}

This bound is probably not tight, most likely the log factor can be removed.
If we only have the "separated" condition, we have a weaker bound, surprisingly, it is also tight.

\begin{theorem}[\cite{KPTU21}]\label{sep} 
Suppose that $M$ is a separated topological multigraph and
$e(M)\ge 4n(M)$.
Then $$\cro(M)\ge c\frac{e(G)^{2.5}}{n(G)^{1.5}}$$
for some constant $c>0$. This bound is asymptotically tight. 
\end{theorem}

We show why 
this bound it asymptotically tight. For any $m$ (divisible by $3$), 
let $H_m$ be the following topological multigraph. 
Let $V(H)=V_1\cup V_2\cup V_3$, each of size $m/3$. 
Place the points of $V_i$ on the line $y=i$.
Connect each point in $V_1$ to each point in $V_3$ with $m/3$ parallel edges, 
crossing the line $y=2$ between different consecutice points of $V_2$.
We can do it such that any two edges cross at most twice. This multigraph has $m$ vertices, 
$m^3/27$ edges and less than $m^6$ crossings.
Suppose now that $n>0$ and $4n<e<cn^3$.
Take roughly $n^{3/2}/e^{0.5}$
disjoint copies of $H_m$ with $m\approx\sqrt{e/n}$. 
This topological multigraph $M$ has $n$ vertices, $e$ edges and 
$\cro(M)=\Omega\left(\frac{e(G)^{2.5}}{n(G)^{1.5}}\right)$.

\smallskip

To prove these theorems, the most natural idea 
is to adapt the proof of the Crossing Lemma, Theorem \ref{crossinglemma} in Section \ref{sec:crlemma} to bound the number of crossings in separated, (single-crossing, locally starlike) multigraphs. 

It follows from Euler's formula, that a separated topological multigraph of $n$ vertices and no 
edge crossings
has at most $3n-6$ edges. 
Observe, that if we delete an edge from a separated topological multigraph, it remains separated.
Therefore, we can apply the induction in the first part of the proof of the 
Crossing Lemma, \ref{crossinglemma} and get that 
for any separated topological multigraph $M$, $\cro(M)\ge e(M)-3n(M)$.

But there is a problem with the second step. 
When we take a random induced subgraph, it might not be separated.
Indeed, suppose that vertices $u$ and $v$ are connected by edges $\alpha$ and $\beta$.
By the conditions, these edges together determine a simple closed curve, which contains some vertices 
$w_1, \ldots, w_k$. When we take a random set of vertices, it is possible that we keep $u$ and $v$, 
but delete $w_1, \ldots w_k$. 
Despite many attempts, it is not known how to modify the random process to rule out this scenario.
Now we sketch another proof of the Crossing Lemma, which can be used in these cases as well.

\begin{theorem}[Weak Crossing Lemma]\label{weakcrossinglemma}
For any simple graph $G$, if
$e(G)\ge 4n(G)$ 
then $$\cro(G)\ge 10^{-26}\frac{e(G)^3}{n(G)^2}.$$ 
\end{theorem}

For the (sketch of the) proof, we need some preparation.
The {\em bisection width}, $b(G)$ of a graph $G$ is the minimum number of edges whose removal splits the graph into
two, roughly equal subgraphs. More precisely, $b(G)$ is the minimum number of edges of $G$ 
between $V_1$ and $V_2$ over all partitions $V(G)=V_1\cup V_2$, $|V_1|, |V_2|\ge |V(G)|/3$.
Based on the Lipton-Tarjan separator theorem \cite{LT79} for planar graphs, and its generalization by Leighton \cite{L83}, Pach, Shahrokhi and Szegedy \cite{PSS96} proved the following relationship between the bisection width and the crossing number.

\begin{theorem}[\cite{PSS96}]\label{PSS}
Suppose that  $G$ is a graph of $n$ vertices of degrees $d_1, \ldots, d_n$. Then
$$b(G)\le 10\sqrt{\cro(G)}+10\sqrt{\sum_{i=1}^nd_i^2}.$$
\end{theorem}

\begin{proof}[Sketch of the proof of the Weak Crossing Lemma]
Let $G$ be a graph of $n$ vertices and $e\ge 4n$ edges. 
If $e\le 10^6n$, then the statement follows from equation (\ref{e-3n}). Assume for a contradiction that $e>10^6n$ and
$\cro(G)< 10^{-26}\frac{e(G)^3}{n(G)^2}$.  

By splitting large degree vertices of $G$ we can get a graph with the same number of edges, at most twice as many vertices and with maximum degree at most
twice the average. Here, for simplicity, we assume 
that $G$ already has this property, its maximum degree is at most twice the average degree. This is important when we apply Theorem \ref{PSS} and want to control the term $\sqrt{\sum_{i=1}^nd_i^2}$.
By a repeated application of Theorem \ref{PSS} we break $H$ into smaller parts. We apply the following procedure. 

\bigskip

\noindent {\sc Decomposition Algorithm}

\medskip

\noindent {\sc Step 0.} {\bf Let} $G^0=G$, 
$G^0_1=G$, $M_0=1$, $m_0=1$.
Suppose that we have already executed {\sc Step} $i$, and
the resulting graph, $G^{i}$, consists of $M_i$ parts,
$G_1^{i}, G_2^{i}, \ldots , G_{M_{i}}^{i}$, each having at most
$(2/3)^in$ vertices. Assume without loss of generality
that the first $m_i$ parts of $G^i$ have at least
$(2/3)^{i+1}n$ vertices and the remaining $M_i-m_i$ have fewer. 

\smallskip

\noindent {\sc Step $i+1$.} {\bf If}
\begin{equation}\label{egyes}
(2/3)^{i} < \frac{e}{n^2}, 
\end{equation}
{\bf then} {\sc stop}. (\ref{egyes}) is called the {\it stopping
rule}.

{\bf Else}, for $j=1, 2, \ldots , {m_{i}}$, 
apply Theorem \ref{PSS} to $G_j^{i}$ and 
delete ${b}(G_j^{i})$
edges from it such that $G_j^{i}$ falls into
two parts, each of at most
$(2/3)n(G_j^{i})$ vertices. Let $G^{i+1}$ denote the resulting graph
on the original set of $n$ vertices. Clearly, each part of $G^{i+1}$ has at
most $(2/3)^{i+1}n$ vertices.

\bigskip

Suppose that the {\sc Decomposition Algorithm} terminates
in {\sc Step} $k+1$.
Since we assumed that $\cro(G)< 10^{-26}\frac{e(G)^3}{n(G)^2}$, 
components $G_j^i$ also have low crossing numbers. Therefore, when we apply
Theorem \ref{PSS}, we do not remove too many edges. 
It can be verified by a rather technical but straightforward
calculation \cite{PST00, KPTU21} 
that the total number of edges removed in the procedure is 
at most $e/2$. 

So for the number of edges of the graph $G^k$ obtained
in the final step of the {\sc Decomposition Algorithm} we have 
$e(G^k)> \frac{e}{n^2}$. 

On the other hand, 
we know by the stopping rule
that each component of $G^k$ has at most $e/n$ vertices. Therefore, even if each component is a complete graph, $G^k$ has less than $e/2$ edges
which is the desired contradiction.
This concludes the proof. The details can be found in \cite{PST00} and in \cite{KPTU21}. \hfill
\end{proof}

This proof of the Crossing Lemma 
appeared first in the paper of Pach, Spencer and T\'oth \cite{PST00} in a more general context.

A graph property ${\cal P}$ is {\em monotone} if the following two conditions are satisfied: (i) if $G$ satisfies ${\cal P}$, then every subgraph of $G$ satisfies ${\cal P}$ as well, and (ii) if $G_1$ and $G_2$ satisfy ${\cal P}$, then their disjoint union also satisfies ${\cal P}$.
Let ${\rm ex}(n, {\cal P})$ be the maximum number of edges of a simple graph 
of $n$ vertices that satisfy property ${\cal P}$.

\begin{theorem}[\cite{PST00}]\label{monotonecr}
Let ${\cal P}$ be a monotone graph property with 
${\rm ex}(n, {\cal P})=O(n^{1+\alpha})$. 
Then there are two constants, $c, c'>0$ such that the crossing number of any simple graph $G$ with property ${\cal P}$ with $n$ vertices and $e>cn\log^2n$ edges satisfies
$$\cro(G)\ge \frac{e^{2+1/\alpha}}{n^{1+1/\alpha}}.$$

If ${\rm ex}(n, {\cal P})=\Theta(n^{1+\alpha})$, then this bound is asymptotically sharp.
\end{theorem}

When ${\cal P}$ is just the property that the graph is simple, we get the Crossing Lemma.  
Very recently, Chen and Ma \cite{CM25} proved a slightly stronger
version of Theorem \ref{monotonecr}, they replaced the condition
$e\ge cn\log^2n$ by the weaker condition $e\ge cn$. 
The proof is a more general version of the proof of the Weak Crossing Lemma above, using the bisection width.
We show now why is it asymptotically tight when ${\rm ex}(n, {\cal P})=\Theta(n^{1+\alpha})$.
The construction is basically the same as in the Crossing Lemma or in Theorem \ref{sep}. We take a disjoint union of 
extremal graphs of proper sizes. 
Let $H_m$ be a graph of $m$ vertices, $Cm^{1+\alpha}$ edges, satisfying Property ${\cal P}$.  Clearly, $\cro(H_m) < C^2m^{2+2\alpha}$. 
Let $n>0$, $e>cn$ and $e= O(n^{1+\alpha})$. 
Let $m=(e/Cn)^{1/\alpha}$ and let $G$ be $n/m$ disjoint copies of $H_m$.
This graph has $n$ vertices, roughly $e$ edges and 
$\cro(G)= O\left(\frac{e^{2+1/\alpha}}{n^{1+1/\alpha}}\right)$.

Let ${\cal P}$ be the property that the graph does not contain some fixed graphs as subgraphs.
This is a monotone property, so 
Theorem \ref{monotonecr} can be applied and one can get very interesting results \cite{PST00}. Here is a prominent example, when the forbidden subgraph is a $C_4$. 
It is known that  ${\rm ex}(n, C_4)\approx n^{3/2}/2$ 
\cite{ERS66}. 

\begin{corollary}
Suppose that $G$ does not contain $C_4$ as a subgraph and $e(G)\ge 4n(G)$. Then 
$$\cro(G)\ge 10^{-12}\frac{e^{4}}{n^{3}}$$
and this bound is tight apart from the value of the constant.
\end{corollary}

The bisection method was used again by Kaufmann, Pach, T\'oth and Ueckerdt \cite{KPTU21}
to prove a common generalization of Theorems \ref{sep+sing+lsl}, \ref{sep+lsl}, \ref{sep+sing}, \ref{sep}, \ref{monotonecr}.

A {\em drawing style} ${\cal D}$ is a property of topological multigraphs or drawings.
A topological multigraph is in drawing style ${\cal D}$ if it satisfies 
property ${\cal D}$.

Let ${\cal D}$ be a drawing style and let $M$ be a topological 
multigraph in drawing style ${\cal D}$.
The ${\cal D}$-bisection width $b_{\cal D}(M)$ of $M$ is the minimum number of edges
whose removal splits $M$ into two  
topological 
multigraphs, $M_1$ and $M_2$  in drawing style ${\cal D}$, and 
$|V(M_1)|, |V(M_2)|\ge |V(M)|/5$.

A {\em vertex split} in a topological multigraph is the following operation. 
Replace a vertex $v$ by two vertices, $v_1$ and $v_2$, very close to it, and modify edges adjacent to $v$ in a small neighborhood of $v$ so that they are connected to either $v_1$ or $v_2$, 
without creating additional crossings.
A drawing style ${\cal D}$
is {\em monotone}, if removing edges keeps the drawing in ${\cal D}$. 
It is {\em split-compatible}, is vertex splits 
keep the drawing in ${\cal D}$. The maximum degree in a topological multigraph $M$ is 
denoted by $\Delta(M)$.

\begin{theorem}[\cite{KPTU21}]\label{generalizedcrlemma}
{\bf Generalized Crossing Lemma}
Suppose that ${\cal D}$ is a monotone and split-compatible drawing style and there are constants 
$k_1, k_2, k_3, a> 0$ such that the following properties hold for any topological multigraph $M$ in drawing style ${\cal D}$. 

(i) if $\cro(M)=0$, then $e(M)\le k_1n(M)$,

(ii) $b_{\cal D}\le k_2\sqrt{\cro(M)+\Delta(M)e(M)+n(M)}$,

(iii) $e(M)\le k_3n(M)^{1+a}$.

Then there is a constant $\alpha>0$ such that for any topological multigraph $M$ in drawing style ${\cal D}$, with $e(M)\ge (k_3+1)n(M)$ we have
$$\cro(G)\ge\alpha\frac{e(M)^{2+1/a}}{n(M)^{1+1/a}}.$$

\end{theorem}

\section{Non-homotopic drawings}\label{sec:nonhomotop}

What is the "weakest possible" condition that guarantees many crossings in terms of the number of edges?
We certainly have to rule out the trivial 
example from the previous section, that is, arbitrarily many loops or parallel edges, without crossings.
A natural candidate is that we do not allow that an edge can be continuously deformed into another edge 
without passing through vertices. 
More precisely, 
consider the drawing on a vertex-punctured plane (but it could be the sphere as well) 
that is, vertices are removed, and we assume that no two parallel edges are homotopic.

In this case there is again a positive answer, shown by Pach, Tardos and T\'oth \cite{PTT22}.
A multigraph $G$, drawn in the plane, punctured at the vertices,  
with no homotopic edges is called a {\em non-homotopic topological multigraph}. If it is clear from the context, we do not make a notational distinction between 
a non-homotopic topological multigraph and its underlying abstract multigraph.
The {\em crossing number} $\cro(G)$
of a non-homotopic topological multigraph $G$ 
is simply the number of edge crossings in its drawing.
Define the function ${\rm cr}(n,m)$ as the minimum crossing number
of a non-homotopic multigraph with $n$ vertices and $m$ edges.

\begin{theorem}[\cite{PTT22}]\label{lowerbound1}
The crossing number of a non-homotopic topological multigraph $G$ with $n>1$ vertices and $m>5n$ edges satisfies ${\rm cr}(G)\ge\frac1{40}\frac{m^2}{n}$. That is, ${\rm cr}(n,m)\ge\frac1{40}\frac{m^2}{n}$. 
\end{theorem}

This bound is tight up to a polylogarithmic factor.

\begin{theorem}[\cite{PTT22}]\label{upperbound}
  For any
$n\ge 2$, $m>4n$,
there exists a non-homotopic topological multigraph
  $G$ with $n$ vertices and $m$ edges
  such that its crossing number satisfies
${\rm cr}(G)\le100\frac{m^2}{n}\log_2^2\frac mn$. In other words, 
 ${\rm cr}(n,m)\le 100\frac{m^2}{n}\log_2^2\frac mn$. 
\end{theorem}

Theorem \ref{lowerbound1} essentially says that if the number of non-homotopic edges is large, then an average pair of edges 
intersects $1/12n$ times. But it is natural to conjecture 
that the more edges we have the more intersections we need on an average pair of edges. That is, Theorem \ref{lowerbound1} can be improved. The following result is the first step in this direction.

\begin{theorem}[\cite{PTT22}]\label{lowerbound2}
The minimum crossing number of a non-homotopic multigraph with
$n\ge 2$ vertices and $m$ edges is super-quadratic in $m$.
That is, for any fixed $n\ge 2$, we have
 $$\lim_{m\rightarrow\infty}\frac{{\rm cr}(n,m)}{m^2}=\infty.$$
Explicitly, we have
\begin{equation}
  \frac{{\rm cr}(n,m)}{m^2} =
    \begin{cases}
       \Omega\left(\frac{(\log m/\log\log m)^{1/6}}{n^8}\right) & \text{for $m>4n$},\\
       \Omega\left({\log^{2/3}m}\right)& \text{for fixed $n$.}
    \end{cases}
\end{equation}
\end{theorem}

Recently, Hubard and Parlier \cite{HP25} improved Theorem \ref{lowerbound2} and showed that the bound in Theorem \ref{upperbound} is asymptotically tight. Their proof uses the geometry of hyperbolic surfaces.

\begin{theorem}[\cite{HP25}]\label{lowerbound3}
The minimum crossing number of a non-homotopic multigraph with
$n\ge 2$ vertices and $m$ edges satisfies
 
\begin{equation}
  \frac{{\rm cr}(n,m)}{m^2} =
\Omega\left(\log^2(m/n)-256n\right) 
\end{equation}
\end{theorem}

 A topological multigraph is called {\em $k$-crossing}
if any two (one)  edges (edge) 
can intersect (itself) at most $k$ times. 
Estimating the minimum number of crossings in a non-homotopic multigraph turned out to be 
closely related to the following 
well studied topological problem.

\smallskip
\begin{problem}\label{pro}
{What is the maximum number $f(n, k)$ of edges of a non-homotopic $k$-crossing multigraph of $n$ vertices?} 
\end{problem}

\smallskip

This problem, together with many  related problems, has a long and interesing history,  
\cite{JMM96, AS18, GISW24, ABG19, D22, G19, MRT14, P15, PS19, M08, HP24}.


Juvan, Malni\v c and Mohar \cite{JMM96} proved that   $f(n, k)$ exists for any $n, k>0$.
They showed that
\begin{equation}\label{jmmbound}
f(n,k)\le (nk)^{O(nk^2)}
\end{equation}

It was improved very recently by Gir\~ao, 
Illingworth, Scott, and Wood \cite{GISW24} to

\begin{equation}\label{giswbound}
f(n,k)\le 6^{13n(k+1)}.
\end{equation}

Aougab and Souto \cite{AS18} showed that for any fixed $n$, 

\begin{equation}\label{asbound}
f(n,k)\le 2^{O(\sqrt{k})}.
\end{equation}

Juvan, Malni\v c, Mohar \cite{JMM96} and 
Aougab, Souto \cite{AS18} studied the problem in 
a more general context, not only on the plane but also on other surfaces.

Pach, Tardos and T\'oth \cite{PTT22} noticed that any upper bound on 
$f(k, n)$, for every $k, n$, implies that
 $$\lim_{m\rightarrow\infty}\frac{{\rm cr}(n,m)}{m^2}=\infty.$$
 They proved a very weak upper bound on 
$f(n, k)$, namely, $f(n,k)\le 2^{(2k)^{2n}}$ which appeared in the preliminary
version of their paper \cite{PTT22}. This implies that
$$\frac{{\rm cr}(n,m)}{m^2} =
\Omega\left(\frac{(\log m)^{1/6n}}{n^7}\right).$$
Then they learned about the above mentioned, much stronger results.
The bound (\ref{jmmbound}) of Juvan, Malni\v c, Mohar \cite{JMM96} implies that
$$\frac{{\rm cr}(n,m)}{m^2} =
\Omega\left(\frac{(\log m/\log\log m))^{1/6}}{n^8}\right).$$
The result (\ref{giswbound}) of 
Gir\~ao, 
Illingworth, Scott, and Wood (\cite{GISW24}  implies that 
$$\frac{{\rm cr}(n,m)}{m^2} =
\Omega\left(\frac{(\log m)^{1/3}}{n^{10/3}}\right).$$
The proofs of these three upper bounds are similar, first they construct a "frame" from some of the edges and
investigate how the other edges can cross this "frame".

Finally, for any fixed $n$, it follows from the bound of 
Aougab and Souto (\ref{asbound}) that 
$$\frac{{\rm cr}(n,m)}{m^2} =
\Omega\left(\log^{2/3}m\right).$$



The best known general lower bound on $f(n,k)$ is the following. 

\begin{theorem}[\cite{PTT22}]\label{looplowerbound}
For any $k\ge n\ge 2$ integers we have
$$f(n,k)\ge2^{\sqrt{nk}/3}.$$
\end{theorem}

This result asymptotically  matches the upper 
bound of Aougab and Souto (\ref{asbound}) for any fixed $n$. It also implies Theorem \ref{upperbound}. 

\bigskip

Now we sketch the proofs Theorems \ref{lowerbound1}, \ref{upperbound} and \ref{looplowerbound}. 
The full proofs can be found 
in \cite{PTT22}. 
We start with some preparation.


\begin{figure}[!ht]
\begin{center}
\scalebox{0.5}{\includegraphics{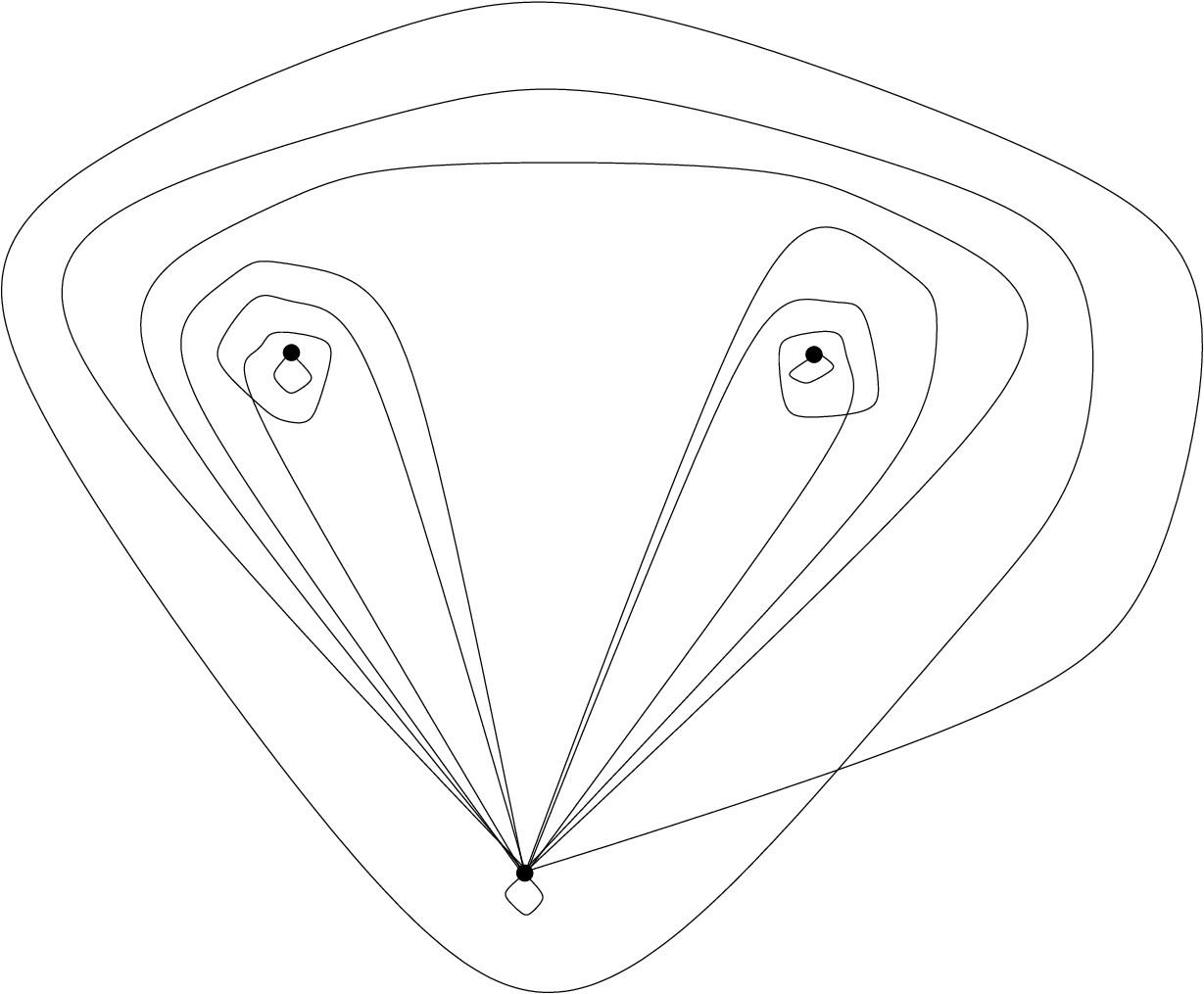}}
\caption{A non-homotopic loose multigraph in the plane with $3$ vertices and $9$
edges, all of which are loops.}\label{9loopplane}
\end{center}
\end{figure}

\smallskip

We say that a topological multigraph is {\em loose} if no pair of {\em distinct}
edges cross each other. An edge can cross itself. 
We start by finding the maximum number of edges in
a loose non-homotopic multigraph in the plane.
We will see that despite allowing parallel edges, loops,
and self-intersections, loose non-homotopic multigraphs
with $n>2$ vertices have at most $4n-3$ edges, just slightly more than planar graphs. 
The proof is easier if we work on the sphere instead of the plane first.
A non-homotopic multigraph on the sphere is defined analogously to those in the plane.
We show that the maximum number of edges of a loose 
non-homotopic multigraph on the sphere
with $n>2$ vertices is $4n-6$.

A loop is trivial if it is contractible to a point. Clearly, a non-homotopic multigraph (in the plane or sphere)
contains at most $n$ trivial loops, one at each vertex. 
Note that here we use a slightly different terminology than in \cite{PTT22}. In \cite{PTT22} trivial loops were
not allowed. Therefore, the bounds and calculations are slightly different.

\begin{lemma}\label{claim0}
On the sphere, any loose non-homotopic multigraph with $n>2$ vertices has $m\le4n-6$ edges.
For $n=2$, the maximum number of edges is $3$.
\end{lemma}

\begin{proof}
Clearly, in a maximal loose non-homotopic multigraph there are exactly $n$ trivial loops, one at each vertex, therefore, 
the following statement is equivalent to the Lemma.

\begin{proposition}\label{prop-sphere}
{On the sphere, any loose non-homotopic multigraph with $n>2$ vertices and no trivial loops has $m\le3n-6$ edges.
For $n=2$, the maximum number of edges is $1$.}
\end{proposition}

Now we prove Proposition \ref{prop-sphere}. We omit the case $n=2$ here, so let $n>2$.

Assume for a contradiction that there is a loose non-homotopic multigraph $H$ on the sphere, which is a counterexample to the proposition. 
Let $H$ to be a counterexample with {\em minimum} number of  crossings, and among those, the maximum number of edges.
Let $n$ and $m$ be the number of vertices and edges of $H$. 

\smallskip

Suppose  first that there is no edge crossing in $H$, so we have a {\em plane} drawing.
Then we can also assume that $H$ is connected, otherwise two components could be
joined by an extra edge without creating a crossing.
If each face has at least three edges on its boundary (counted with multiplicity) then a standard application of Euler's formula shows that $m\le 3n-6$, a contradiction.
Therefore, $H$ must have a face bounded by one or two edges. Consequently, this face 
is bounded by a trivial loop or two parallel edges, or one edge on both sides. 
The first two possibilities contradict the assumptions. 
In the last case we have 
$n=2$ which is again a contradiction.

\smallskip

For the rest of the proof we can assume that 
$H$ has at
least one self-intersecting edge $e$.
Find a minimal interval $\gamma$ of $e$ between two
occurrences of the same intersection point $p$.
This simple closed curve partitions the sphere $S^2$ into two connected components.
We call them (arbitrarily) the \emph{left} and \emph{right sides}
of $\gamma$. 
Obviously, $e$ is the only edge that may run between these sides.
Let $H_1$ and $H_2$ be the subgraphs of $H-e$
induced by the vertices in the left and right sides of
$\gamma$, respectively. Both of them are loose topological multigraphs,
but they may contain homotopic edges, and a trivial loop.
Add $p$ to both of them as an isolated vertex, now both of them are 
non-homotopic with no trivial loops.

If an endpoint $u$ of $e$ lies in the left part,
then by adding to $H_1$ a non-self-intersecting edge
connecting $p$ and $u$ along $e$, we create no new
intersection and do not violate the non-homotopic condition either.
The resulting topological multigraph $H'_1$ is a loose
non-homotopic multigraph on the sphere with $n_1$ vertices and $m_1$ edges.
Analogously, we can construct the loose non-homotopic multigraph
$H'_2$ from $H_2$. Denote its number of vertices and edges by
$n_2$ and $m_2$, respectively. We have $n_1+n_2=n+2$
and $m_1+m_2\ge m$.
We eliminated a self-crossing (of $e$) and did not add any new crossings, so
the number of crossings in both $H'_1$ and $H'_2$ are smaller than 
the number of crossings in $H$. 
\smallskip

If $n_1, n_2>2$, then we have $m_1\le3n_1-6$ and $m_2\le3n_2-6$,
by the minimality of $H$. Summing up these inequalities, we get $m\le3n-6$, contradicting our assumption that $H$ was a counterexample.

If $n_1=1$ or $n_2=1$, then all vertices of $H$ lie on the same side of
$\gamma$. In this case, by deleting $\gamma$ from $e$, the homotopy class of $e$ remains the same. Hence, the resulting topological multigraph is still a loose non-homotopic multigraph with $n$ vertices and $m$ edges, but its crossing number is smaller than that of $H$, contradicting the minimality of $H$.
\smallskip

Finally, consider the case $n_1=2$ or $n_2=2$.
By symmetry, we can assume that $n_1=2$, $n_2=n$, so we have a single vertex
$u$ of $H$ on the left side of $\gamma$ and $n-1$ vertices on the right side.
No edge of $H-e$ can lie in the left side since it 
would be a trivial loop. If $e$ has at least one endpoint
in the right part, then we have $m_2=m$. This implies that $H'_2$ is another
counterexample to the lemma with fewer crossings,
contradicting the minimality of $H$.
\smallskip

\vskip0.5cm

\begin{figure}[!ht]
\begin{center}
\scalebox{0.5}{\includegraphics{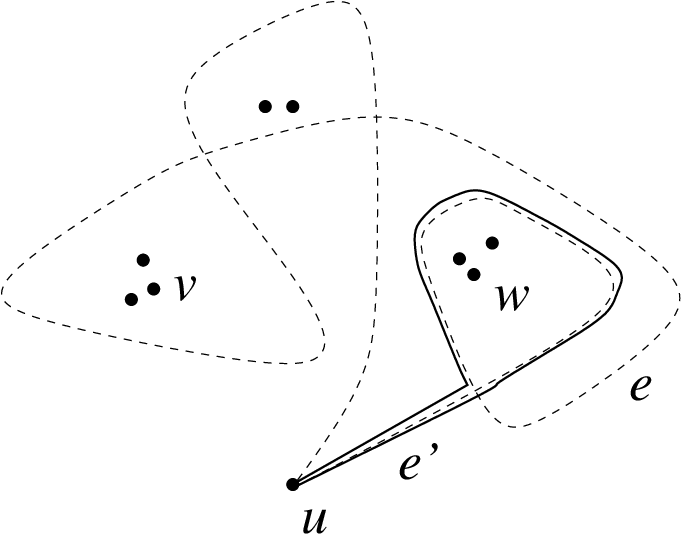}}
\caption{The replacement of $e$ by $e'$ in the proof of 
Proposition \ref{prop-sphere}.}\label{replace}
\end{center}
\end{figure}

\smallskip

Therefore, $e$ must be a loop at $u$.
Since it is not a trivial loop, it 
separates a pair of vertices, $v, w\in V(H)-u$
from each other.
But then we can draw another loop $e'$ 
along some parts of
$e$, with no self-intersection, so that it also separates $v$ and $w$.
Therefore, $e'$ is not trivial either. See Fig. \ref{replace}.

Let $H'$ be the topological multigraph  obtained from $H$ by replacing $e$ by $e'$.
There is no other loop  in $H'$ homotopic to $e'$, because there is no other loop at $u$.
Hence, $H'$ is a loose  non-homotopic
multigraph with no trivial loops, the same number of edges 
but fewer crossings than $H$, 
contradicting the minimality of $H$. 
For the lower bound, for $n\ge 3$, just take a triangulation, which has $3n-6$ edges. 
This concludes the proof of the Proposition \ref{prop-sphere}
and Lemma \ref{claim0}
\hfill  
\end{proof}

\begin{lemma}\label{claim1}
In the plane, the maximum number of edges of a loose non-homotopic multigraph with $n>1$ vertices is $4n-3$. 
\end{lemma}

\begin{proof}
Let $H$ be a loose non-homotopic multigraph in the plane with $n\ge1$ vertices and $m$ edges. Consider the plane as the sphere $S^2$ with a point $p^*$ removed. Add $p^*$ to $H$ as an additional vertex, and the trivial loop at $p^*$ to obtain a 
loose non-homotopic multigraph $H'$ with $n+1$ vertices and $m+1$ edges. If $n>1$, applying Lemma~\ref{claim0} to $H'$, we obtain that $m\le 4n-3$, as required. To see that the bound is tight, take a triangulation with $n$ vertices and $3n-6$ edges. Let $uvw$ be the boundary of the unbounded face. Add another non-self-intersecting edge connecting $u$ and $v$ in the unbounded face, which is not homotopic with the arc $uv$ of $uvw$. Add two further loops at $u$. First, a simple loop $l$ that has all other edges and vertices (except $u$) in its interior, and then another loop $l'$  outside of $l$, which goes twice around $l$. (Of course, $l'$ must be self-intersecting.) And finally, add a trivial loop at each vertex.
\hfill 
\end{proof}



\smallskip

There are many other, different constructions that show that the bounds in 
Lemmas \ref{claim0} and \ref{claim1} are tight.

\smallskip

\begin{proof}[Proof of Theorem~\ref{lowerbound1}]
Let $G$ be a non-homotopic topological
multigraph in the plane with $n>1$ vertices and $m>5n$ edges.

Let $D$ denote the {\em non-crossing graph} of the edges of $G$, that is, let $V(D)=E(G)$ and
connect two vertices of $D$ by an edge if and only if the corresponding edges of $G$
do not share an interior point. Any clique in $D$ corresponds to a loose non-homotopic sub-multigraph of $G$. Therefore, by Lemma~\ref{claim1}, $D$ has no clique of size $4n-2$.
Thus, by Tur\'an's theorem~\cite{T41},
$$|E(D)|\le\frac{|V(D)|^2}{2}\left(1-\frac1{4n-3}\right)=\frac{m^2}{2}\left(1-\frac1{4n-3}\right).$$
The crossing number ${\rm cr}(G)$ is at least the number of crossing pairs of edges in $G$, 
which is equal to the number of non-edges of $D$. Since $m>5n$, we have
$${\rm cr}(G)\ge\binom m2-\frac{m^2}{2}\left(1-\frac1{4n-3}\right)<\frac{m^2}{8}\left(\frac1{n-1}-\frac4{m}\right)
>\frac{m^2}{8}\left(\frac1{n}-\frac4{5n}\right)=\frac1{40}\frac{m^2}{n},$$
as claimed. \hfill   
\end{proof}

The proof above gives a lower bound on the number
of crossing {\em pairs of edges} in $G$,
and in this respect it is tight up to a constant factor.
Suppose that $n$ is even and $m$ is divisible by $n$.
Let $G_0$ be a non-homotopic topological multigraph with
two vertices and $\frac{2m}n$ non-homotopic loops on one of its vertices.
Taking $\frac{n}2$ disjoint copies of
$G_0$, we obtain a non-homotopic topological
multigraph with $n$ vertices, $m$ edges, and fewer than $\frac{m^2}{n}$
crossing pairs of edges.

\smallskip

Now we construct many, pairwise homotopic loops in the punctured plane such that any two intersect at most  $k$ times. 

\smallskip

\begin{proof}[Proof of Theorem~\ref{looplowerbound}]
Let $S=\RR^2\setminus\{a_1,\dots,a_n\}$, where $a_1,\ldots, a_n$ are distinct points in the plane, and let $x\in S$ be also fixed. Assume without loss of generality that $a_i=(i, 0),\; 1\le i\le n,$ and $x=(0, -1)$. An {\em $x$-loop} is a (possibly self-crossing)
oriented path in $S$ from $x$ to $x$, i.e., a continuous function
$f\colon[0,1]\to S$ with $f(0)=f(1)=x$.

The homotopy group of $S$ is the free group $F_n$ generated by $g_1,\dots,g_n$, where $g_i$ can be represented by a triangular $x$-loop around $a_i$, for example the one going from $x$ to $(2i-1, 1)$, from here to $(2i+1, 1)$, and then back to $x$ along three straight-line segments; see~\cite{L77}.
\smallskip

We define an {\em elementary loop} to be a polygonal $x$-loop with intermediate vertices $$(1,\varepsilon_1),\,(2, \varepsilon_2),\,\dots,\,(n,\varepsilon_n),\,(n+1,-1),$$ in this order, where each $\varepsilon_i\; (1\le i\le n)$ is equal either to $1/2$ or to $-1/2$. There are $2^n$ distinct elementary loops, depending on the choice of the $\varepsilon_i$. Each of them represents a distinct homotopy class of the form $g_{i_1}\cdots g_{i_t}$, where the indices form a strictly increasing sequence. By making infinitesimal perturbations on the interior vertices of the elementary loops, we can make sure that every pair of them intersect in at most $n-1$ points. Thus, we have $f(n,n)\ge2^n$. See Fig. \ref{elemloop}.

\vskip0.5cm

\begin{figure}[!ht]
\begin{center}
\scalebox{0.5}{\includegraphics{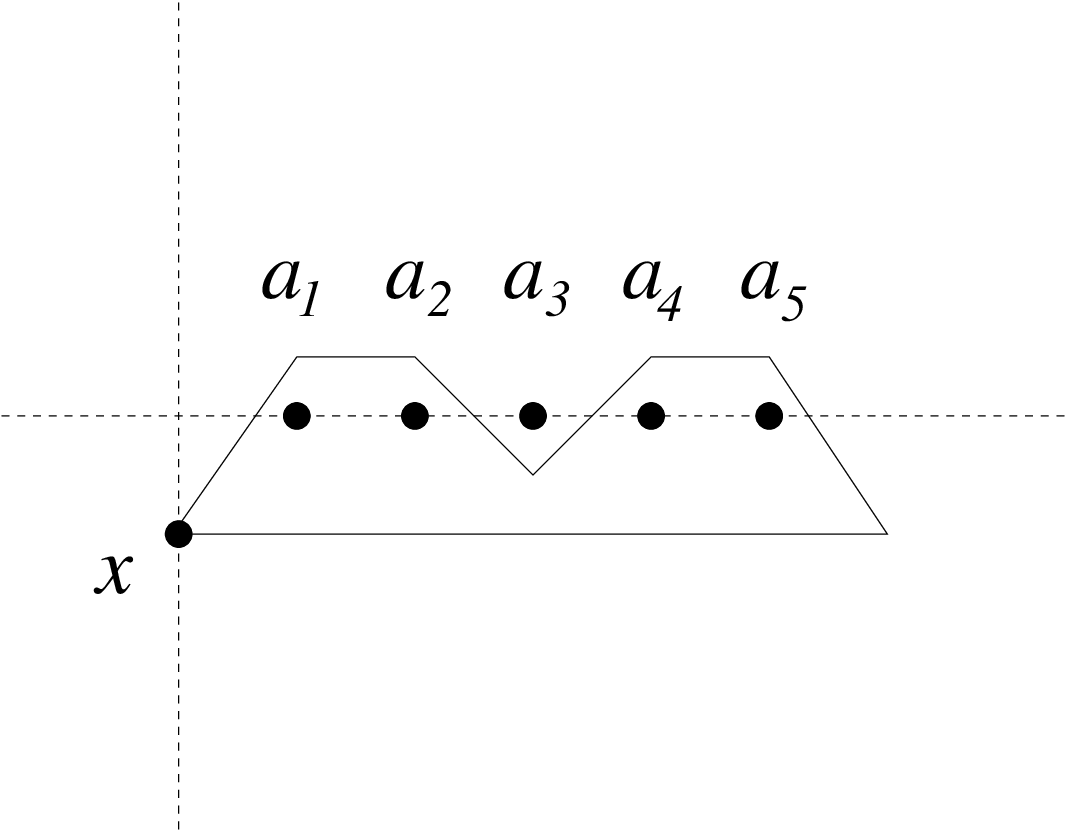}}
\caption{An elementary $x$-loop ($n=5$).}\label{elemloop}
\end{center}
\end{figure}

\smallskip

If $n\le k\le9n$, we have $f(n,k)\ge f(n,n)\ge 2^n\ge2^{\sqrt{nk}/3}$, and we are done.
\smallskip

In the case $k>9n$, we consider all $x$-loops which can be obtained as the product (concatenation) of $j=\left\lfloor\sqrt{\frac{k-1}n}\right\rfloor\ge3$ elementary loops. Unfortunately, some of these concatenated $x$-loops will be homotopic. For example, if the elementary loops $l_1,l_2,l_3,$ and $l_4$ represent the homotopy classes $g_1,g_2g_3,g_1g_2,$ and $g_3$, respectively, then $l_1l_2$ and $l_3l_4$ are homotopic. To avoid this complication, we only use the $2^{n-1}$ elementary loops that represent homotopy classes involving $g_1$ (that is, the ones with $(1,1/2)$ as their first intermediate vertex). Concatenating $j$ such elementary loops, we obtain $2^{j(n-1)}$ different $x$-loops, no pair of which are homotopic. By infinitesimal perturbation of the interior vertices of these $x$-loops (including the $j-1$ interior vertices at $x$), we can ensure that they do not pass through $x$, and no two polygonal paths corresponding to a single elementary loop intersect more than $n$ times. Therefore, any pair of perturbed concatenated loops cross at most $j^2n<k$ times, and the same bound holds for the number of self-intersections of any concatenated loop. This yields that $f(n,k)\ge2^{j(n-1)}\ge 2^{\sqrt{nk}/3}$, completing the proof of the theorem. \hfill  
\end{proof}

\section{Other crossing numbers}\label{sec:othercrossingnumbers}

The crossing number of a graph $G$ is usually defined as
``the minimum number of edge crossings in any drawing of $G$
in the plane" \cite{BL84}. 
But 
this definition can be interpreted in
several different ways. Sometimes it is assumed that in
a drawing two edges have at most one common point, which is a crossing or a common endpoint (\cite{WB78, B91}). Many authors do not
make this assumption (\cite{T70, GJ83, SSSV97}). If two edges are allowed to cross several
times, we may count their intersections with multiplicity or
without \cite{AS14}. We can impose some further
restrictions on the drawings (e.g., the edges must be
straight-line segments \cite{J71}, or polygonal paths of
length at most $k$ \cite{BD93}). No matter what definition
we use, the determination of the crossing number of a graph
appears to be an extremely difficult task (\cite{GJ83, B91}). In fact, we do not even know the 
value of {\em any} of the above quantities for the complete
graph $K_n$ with $n$ vertices and for the complete bipartite
graph $K_{n,n}$ with $2n$ vertices, for large $n$. 

In his famous book about crossing numbers, Schaefer \cite{S18} lists more than 100 different versions of the crossing number.
The systematic study of different versions of the crossing number and their relationships started with the work of
Pach and T\'oth \cite{PT00, PT00a}. See also the dynamic survey of Schaefer \cite{S12}.

The most commonly used version of the crossing number is simply called the {\em crossing number} of a graph $G$, $\cro(G)$, is the minimum number of crossings
(crossing points)
over all drawings of $G$. (Here we assume that no three edges cross at the same point, or these crossings are counted with multiplicity, and not edge goes through a vertex.)
The {\em pair-crossing number}, $\pcr(G)$, is the minimum number of pairs of crossing edges
over all drawings of $G$.
As we have seen in Observation \ref{swap}, 
in an optimal drawing for $\cro(G)$, any two edges cross at most once.
Therefore,
it is not easy to see the difference between these two definitions.
Indeed, there was some confusion in the literature between these two notions,
until the systematic study of their relationship. 
Clearly, $\pcr(G)\le\cro(G)$ for every $G$, and
in fact, we cannot rule out the possibility, that $\cro(G)=\pcr(G)$
for every graph $G$. Undoubtedly,
it is the most interesting open problem in this area. Any progress would be a great step in understanding the crossing number.
From the other direction, after several steps, the best known bound is
$\cro(G)=O(\pcr(G)^{3/2}\log\pcr(G))$ \cite{S18}.

The {\em odd-crossing number}, $\ocr(G)$, is the minimum number
of pairs of edges that cross an {\em odd} number of times,
over all drawings of $G$.
Clearly, for every graph $G$,
$$\ocr(G)\le\pcr(G)\le\cro(G).$$
This definition was partly inspired by the weak Hanani-Tutte theorem \cite{C34, PSS07},
which states that if $\ocr(G)=0$, then $G$ is planar, that is, $\ocr(G)=\pcr(G)=\cro(G)=0$.
It was shown in \cite{PSS07} that for $k=1, 2, 3$, if
$\ocr(G)=k$, then $\ocr(G)=\pcr(G)=\cro(G)=k$.
There are examples where $\ocr$ is different from $\pcr$ and $\cro$,
there is an infinite family of graphs with
$\ocr(G)<0.855\cdot\pcr(G)$ \cite{T08}, and there is also a slightly weaker
($\ocr(G)<0.866\cdot\pcr(G)$), completely different construction \cite{PSS09}.
From the other
direction we only have $\pcr(G)<2\ocr(G)^2$ \cite{PT00a}.

There are further distinctions of these crossing numbers according to the roles of crossings between adjacent edges.
For each of $\cro$, $\pcr$, $\ocr$, there are three natural counting rules:

\medskip

\noindent {\bf Rule $+$:} Only those drawings are considered, where adjacent edges cannot cross.

\smallskip

\noindent {\bf Rule $0$:} Adjacent edges are allowed to cross and their crossings are counted as well.

\smallskip

\noindent {\bf Rule $-$:} Adjacent edges are allowed to cross and their crossings are not counted.

\medskip

{\scriptsize

\setlength{\unitlength}{1cm}
\begin{picture}(10, 3.5)
\put(1, 0.2){

\begin{picture}(9.5, 3)

\put(0, 0){\line(1, 0){8.5}}
\put(0, 1){\line(1, 0){8.5}}
\put(0, 2){\line(1, 0){6.5}}
\put(0, 3){\line(1, 0){8.5}}


\put(0, 0){\line(0, 1){3}}
\put(1.5, 0){\line(0, 1){3}}
\put(4.5, 0){\line(0, 1){3}}
\put(6.5, 0){\line(0, 1){3}}
\put(8.5, 0){\line(0, 1){3}}

\put(0.2, 0.3){Rule {--}}
\put(0.2, 1.3){Rule 0}
\put(0.2, 2.3){Rule +}


\put(7.1, 1.9){$\cro(G)$}
\put(7.1, 0.4){$\cro_{-}(G)$}

\put(4.6, 2.4){$\pcr_{+}(G)$}
\put(4.6, 1.4){$\pcr(G)$}
\put(4.6, 0.4){$\pcr_{-}(G)$}

\put(1.6, 2.4){$\ocr_{*}(G)\le \ocr_{+}(G)$}
\put(1.6, 1.4){$\ocr(G)$}
\put(1.6, 0.4){$\ocr_{-}(G)$}

\end{picture}}
\end{picture}

\medskip
}
\centerline{{\bf Table 1.} Nine versions of the crossing number.}\label{9verzio}

\medskip

Combining these rules with the three crossing numbers, we get nine possibilities.
But
it is easy to see that
$\cro_{+}=\cro$ \cite{PT00a}.
On the other hand, regarding Rule $+$ for the odd-crossing number,
it seems more natural to assume that adjacent edges cross an {\em even number of times} than
to assume that they do not cross at all.
So, let $\ocr_{*}(G)$ be the minimum number of odd-crossing pairs of edges over all drawings of $G$
where adjacent edges cross an even number of times.
Therefore,  we have nine  versions, see Table 1. 
In this table, values do not decrease if we move to the right or up, and it was shown in \cite{PSS09} that
 $\cro(G)<2\ocr_{-}(G)^2$. On the other hand, for the greatest surprise of our community there are graphs $G$, where
$\ocr_{-}(G)<\ocr(G)$ \cite{FPSS11}.

For each of these crossing numbers, we can get rid of self-crossings just like in the proof of Observation \ref{swap}. So we can assume without loss
of generality, that {\em self-crossings are not allowed.}

We show that the basic version of the Crossing Lemma (Theorem \ref{crossinglemma}) holds for all nine versions  of the crossing number.

\begin{theorem}[Crossing Lemma for other crossing numbers]\label{crossinglemma9}
Let $c(G)$ be any of the parameters
{$\cro(G)$},
{$\cro_{-}(G)$}
{$\pcr_{+}(G)$},
{$\pcr(G)$},
{$\pcr_{-}(G)$}, 
{$\ocr_{*}(G)$}, {$\ocr_{+}(G)$},
{$\ocr(G)$},
{$\ocr_{-}(G)$}.
Then 
for any simple graph $G$, if
$e(G)\ge 4n(G)$ 
then $$c(G)\ge \frac{1}{64}\frac{e(G)^3}{n(G)^2}.$$ 
\end{theorem}

For the proof we need the Strong Hanani-Tutte theorem \cite{SS10, C34, T70}.

\begin{theorem}[Strong Hanani-Tutte theorem, \cite{C34}]\label{stronghananitutte}
If a graph $G$ can be  drawn in the plane such that every pair of {\em independent} edges cross an even number of times, then it is planar. 
\end{theorem}

In other words, 
if $\ocr_{-}(G)=0$, then it is planar so all other crossing numbers are also $0$.

\smallskip

\begin{proof}[Proof of the Crossing Lemma for other crossing numbers.] 
The proof is very similar to the proof of 
the original Crossing Lemma (Theorem \ref{crossinglemma}). It goes again in two steps. 
It is enough to prove the statement for the smallest one, {$\ocr_{-}(G)$}.

1. Just like in the original proof, we can prove by induction on $e(G)$ that 
\begin{equation}\label{e-3n-odd-cr}
\ocr_{-}(G)\ge e(G)-3n(G).
\end{equation}

If $e(G)\le 3n(G)$, then the statement holds trivially.
Suppose that $e(G)>3n(G)$ and we already proved the statement for smaller values of $e(G)$. 
By the Strong Hanani-Tutte theorem, we have 
$\ocr_{-}(G)>0$. Remove an edge that  crosses an independent edge an odd number of times and 
use the induction hypothesis.

2. 
Consider a drawing of $G$ with $\ocr_{-}(G)$ independent odd-crossing pairs of edges. Take a random induced subgraph
$G'$ of $G$ by selecting each vertex independently with probability $p$. 
We obtain a random subset $U\subset V(G)$ and let $G'=G(U)$, consider it in the drawing inherited 
from the original drawing of $G$. 
We have 
$E[n(G')]=pn(G)$ and 
$E[e(G')]=p^2e(G)$.
Let $ocr_{-}(G')$ be the number of pairs of odd-crossing independent edges of $G'$.
Then clearly $ocr_{-}(G')\ge\ocr_{-}(G')$. 
The probability that an odd-crossing independent pair is present in the drawing of $G'$ is $p^4$, so 
$E[ocr_{-}(G')]=p^4\ocr_{-}(G)$. 

By  (\ref{e-3n-odd-cr}), $\ocr_{-}(G')\ge e(G')-3n(G')$. 
Just like in the proof of the  Crossing Lemma, we 
take expectations of both sides, 
set $p=4n(G)/e(G)$ and obtain that 
$$\ocr_{-}(G)\ge\frac{1}{64}e(G)^3/n(G)^2.$$ 
\end{proof}

Recall that the improvements of the coefficient in the original Crossing Lemma for $\cro(G)$
are based on the following observation: if $G$ has substantially more edges than $3n-6$, then there is an edge with two, three, or even more crossings. See Section \ref{sec:crlemma} for details.  
Using this idea, the linear inequality (\ref{e-3n}), $\cro(G)\ge e(G)-3n(G)$, has been replaced by stronger ones.
Then the probabilistic argument is applied, analogous to Step 2 in the proof of Theorem \ref{crossinglemma}.
Observe that in this argument we can count only crossings between {\em independent} edges. Such a crossing 
appears in the random subgraph $G'$ with probability $p^4$. A crossing between two {\em adjacent} edges 
appears with probability $p^3$, which is not enough (or in fact, too much) for the calculation.
Fortunately, we have Observation \ref{swap} which guarantees
that in a drawing with the minimum number of crossings, only independent edges cross.

For the pair-crossing number and the odd-crossing number, 
we have Theorem \ref{crossinglemma9}, and for the proof we needed 
the Strong Hanani-Tutte theorem instead of Observation \ref{swap} 
to guarantee crossings between independent edges.
However,  
the proofs for the improvements do not work 
for the pair-crossing number and the odd-crossing number. 
As we will see, we can guarantee more crossings (crossing pairs, or odd-crossing pairs), but they might not be independent.
In Section \ref{sec:crlemma} we defined $m_k(n)$, the maximum number of edges
of a simple graph of $n$ vertices that has a drawing  where every edge contains at most $k$ crossings 
(these are the $k$-planar graphs).
Ackerman and Schaefer \cite{AS14}  defined the following analogue: let 
$m''_k(n)$ be the maximum number of edges
of a simple graph of $n$ vertices that has a drawing  where every edge is crossed by at most 
$k$ other edges (no matter, how many times). By definition, $m''_k(n)\ge m_k(n)$. 
It is trivial that $m''_0(n)=3n-6$, and easy to see that $m''_1(n)=4n-8$ for $n\ge 12$. Ackerman and Schaefer
proved that
$m''_2(n)\le 5n-10$, which is tight for infinitely many values of $n$, 
and $m''_3(n)\le 6n-12$, which is probably not tight.
(For comparision, we have $m_3(n)\le 5.5n-11$, which is tight apart from the additive constant.)
The bounds for $m''_k(n)$ imply that 
\begin{equation}\label{pcr-lin}
\pcr(G)\ge 4e(G)-18n(G)
\end{equation}

Consequently, we also have
\begin{equation}\label{pcr+lin}
\pcr_{+}(G)\ge 4e(G)-18n(G)
\end{equation}

In other words, if we do not allow non-independent crossings, then we have $4e-18n$ independent crossings. And now we can apply the probabilistic argument to get the following result. 

\begin{theorem}[Improved Crossing Lemma for the pair-crossing number, \cite{AS14}]\label{ackerman-schaefer}
For any simple graph $G$, if
$e(G)\ge 6.75n(G)$ 
then $\pcr_{+}(G)\ge \frac{1}{34.2}\frac{e(G)^3}{n(G)^2}$. 
\end{theorem}

For the odd-crossing number, the situation is very similar. 
Let 
$m^{\rm odd}_k(n)$ be the maximum number of edges
of a simple graph of $n$ vertices that has a drawing  where every edge is crossed by at most 
$k$ other edges {\em an odd number of times}.
Clearly, we have $m^{\rm odd}_k(n)\ge m''_k(n)\ge m_k(n)$. 
It follows from the Weak Hanani-Tutte theorem 
that $m^{\rm odd}_0(n)=3n-6$.
Karl and T\'oth \cite{KT23} proved that
$m^{\rm odd}_1(n)\le 5n-9$. They conjecture that this bound is far from the truth, probably 
$m^{\rm odd}_1(n)=4n-8$.
It follows that 

\begin{equation}\label{ocr-lin}
\ocr_{+}(G)\ge \ocr(G)\ge 2e(G)-8n(G).
\end{equation}
and then we obtain

\begin{theorem}[Improved Crossing Lemma for the odd-crossing number, \cite{KT23}]\label{karl-toth}
For any simple graph $G$, if
$e(G)\ge 6n(G)$ 
then $\ocr_{+}(G)\ge \frac{1}{54}\frac{e(G)^3}{n(G)^2}$. 
\end{theorem}

\section{Open problems}\label{sec:openproblems}

1.  According to Theorem \ref{midrange} at the end of Section \ref{sec:crlemma} 
if 
$n\ll e\ll n^2$, then
$$\lim_{n\rightarrow\infty}\kappa(n,e)\frac{n^2}{e^3}=c>0.$$
The constant $c$ is called the 
{\em midrange
crossing constant}, see also the end of Section \ref{sec:crlemma}.
It is not known whether the condition $n\ll e\ll n^2$ can be replaced by a  weaker condition, say, 
$C_1n\le e\le C_2n^2$. It is not hard to see that
$(4+\varepsilon)n<e$ and $(1/2-\varepsilon)n^2>e$ are necessary conditions. 
Find the weakest conditions on $e$ in terms of $n$ such that the 
statement of Theorem \ref{midrange} holds.

We have $0.036\le c\le 0.09$, improve these bounds. In particular, are the upper bound constructions 
in \cite{PT97} and \cite{CSSW20} optimal, apart from lower order terms?

\medskip

\noindent 2. In Section \ref{sec:nonhomotop} we have seen that there is a generalization of the Crossing Lemma
for non-homotopic multigraphs \cite{HP25}. Is there any other natural condition on the drawing that guarantees a
Crossing Lemma-type statement for multigraphs? 

\medskip

\noindent 3. According to the Strong Hanani-Tutte theorem, Theorem  \ref{stronghananitutte} 
if $\ocr_{-}(G)=0$, then it is planar so all other crossing numbers are also $0$.
The statement also holds on the projective plane \cite{PSS09a} and on the torus
\cite{FPS20} but 
does not hold on orientable surfaces of genus at least $4$ \cite{FK19}. 
Does it hold on other surfaces without boundary?

\medskip

\noindent 4. Is it true that 
 $\cro(G) =\pcr(G)$
for every graph $G$?

Or a less ambitious problem, is there a constant $c>0$ such that  $\cro(G)\le c\pcr(G)$
for every graph $G$?
The best bound we know is that 
$\cro(G)=O(\pcr(G)^{3/2}\log\pcr(G)$ \cite{S18}.

\medskip

\noindent  5. Similar problem for the odd-crossing number, 
is there a constant $c>0$ such that  $\cro(G)\le c\ocr(G)$
for every graph $G$?
There are examples where $\ocr$ is different from $\pcr$ and $\cro$,
\cite{T08, PSS09} 
and from the other
direction we only have $\pcr(G)<2\ocr(G)^2$ \cite{PT00a}.

\medskip

\noindent 6. In Section \ref{sec:openproblems} we have seen that Theorem \ref{crossinglemma}, that is, the simplest version of the Crossing Lemma, with constant $1/64$, can be generalized for 
all versions of the crossing number, pair-crossing number, and odd-crossing number in Table  1.
%
However, the improvements so far work only for the "$+$" versions. 

Improve the constant in the Crossing Lemma for the   other versions 
of the crossing numbers in Table \ref{9verzio}.

\bigskip

\noindent {\bf Acknowledgement.} We are very grateful to D\"om\"ot\"or P\'alv\"olgyi for his helpful comments, suggestions.

\bibliographystyle{alpha}
\bibliography{references}

\newcommand{\etalchar}[1]{$^{#1}$}
\begin{thebibliography}{ADFM{\etalchar{+}}20}

\bibitem[ABG19]{ABG19}
T.~Aougab, I.~Biringer, and J.~Gaster.
\newblock Packing curves on surfaces with few intersections.
\newblock {\em International Mathematics Research Notices}, 16:5205--5217,
  2019.

\bibitem[ABK{\etalchar{+}}17]{ABK17}
P.~Angelini, M.~Bekos, M.~Kaufmann, M.~Pfister, and T.~Ueckerdt.
\newblock Beyond-planarity: density results for bipartite graphs.
\newblock {\em arXiv preprint}, 2017.

\bibitem[ACFM{\etalchar{+}}12]{ACF12}
B.~Abrego, M.~Cetina, S.~Fern{\'a}ndez-Merchant, J.~Lea{\~n}os, and G.~Salazar.
\newblock On $\le k$-edges, crossings, and halving lines of geometric drawings
  of $k_n$.
\newblock {\em Discrete and Computational Geometry}, 48(1):192--215, 2012.

\bibitem[Ack19]{A19}
E.~Ackerman.
\newblock On topological graphs with at most four crossings per edge.
\newblock {\em Computational Geometry}, 85:101574, 2019.

\bibitem[ACNS82]{ACNS82}
M.~Ajtai, V.~Chv{\'a}tal, M.~Newborn, and E.~Szemer{\'e}di.
\newblock Crossing-free subgraphs.
\newblock In {\em Theory and Practice of Combinatorics}, volume~60 of {\em
  North-Holland Mathematical Studies}, pages 9--12. North-Holland,
  Amsterdam-New York, 1982.

\bibitem[ADFM{\etalchar{+}}20]{ADF20}
O.~Aichholzer, F.~Duque, R.~Fabila-Monroy, O.~Garc{\'i}a-Quintero, and
  C.~Hidalgo-Toscano.
\newblock An ongoing project to improve the rectilinear and the pseudolinear
  crossing constants.
\newblock {\em Journal of Graph Algorithms and Applications}, 24(3):421--432,
  2020.

\bibitem[{\'A}P22]{AP22}
P.~{\'A}goston and D.~P{\'a}lv{\"o}lgyi.
\newblock An improved constant factor for the unit distance problem.
\newblock {\em Studia Scientiarum Mathematicarum Hungarica}, 59(1):40--57,
  2022.

\bibitem[AS14]{AS14}
E.~Ackerman and M.~Schaefer.
\newblock A crossing lemma for the pair-crossing number.
\newblock In {\em International Symposium on Graph Drawing}, pages 222--233,
  Berlin, Heidelberg, 2014. Springer.

\bibitem[AS18]{AS18}
T.~Aougab and J.~Souto.
\newblock Counting curve types.
\newblock {\em American Journal of Mathematics}, 140(6):1423--1441, 2018.

\bibitem[AZ99]{AZ99}
M.~Aigner and G.~Ziegler.
\newblock {\em Proofs from the Book}.
\newblock Springer, Berlin, 1999.

\bibitem[BD93]{BD93}
D.~Bienstock and N.~Dean.
\newblock Bounds for rectilinear crossing numbers.
\newblock {\em Journal of Graph Theory}, 17:333--348, 1993.

\bibitem[Bie91]{B91}
D.~Bienstock.
\newblock Some provably hard crossing number problems.
\newblock {\em Discrete and Computational Geometry}, 6:443--459, 1991.

\bibitem[BK24]{BK24}
A.~B{\"u}ngener and M.~Kaufmann.
\newblock Improving the crossing lemma by characterizing dense 2-planar and
  3-planar graphs.
\newblock {\em arXiv preprint}, 2024.

\bibitem[BL84]{BL84}
S.~N. Bhatt and F.~T. Leighton.
\newblock A framework for solving vlsi graph layout problems.
\newblock {\em Journal of Computer and System Sciences}, 28:300--343, 1984.

\bibitem[BLS19]{BLS19}
J.~Balogh, B.~Lidicky, and G.~Salazar.
\newblock Closing in on hill's conjecture.
\newblock {\em SIAM Journal on Discrete Mathematics}, 33(3):1261--1276, 2019.

\bibitem[Cho34]{C34}
C.~Chojnacki.
\newblock {\"U}ber wesentlich unpl{\"a}ttbare kurven im dreidimensionalen
  raume.
\newblock {\em Fundamenta Mathematicae}, 23:135--142, 1934.

\bibitem[CM25]{CM25}
K.~Chen and J.~Ma.
\newblock On a conjecture of pach-spencer-t{\'o}th for graph crossing numbers.
\newblock {\em arXiv preprint}, 2025.

\bibitem[CSSW20]{CSSW20}
{\'E}.~Czabarka, I.~Singgih, L.~Sz{\'e}kely, and Z.~Wang.
\newblock Some remarks on the midrange crossing constant.
\newblock {\em Studia Scientiarum Mathematicarum Hungarica}, 57(2):187--192,
  2020.

\bibitem[Dey98]{D98}
T.~K. Dey.
\newblock Improved bounds for planar k-sets and related problems.
\newblock {\em Discrete and Computational Geometry}, 19:373--382, 1998.

\bibitem[DLM19]{DLM19}
W.~Didimo, G.~Liotta, and F.~Montecchiani.
\newblock A survey on graph drawing beyond planarity.
\newblock {\em ACM Computing Surveys}, 52(1):1--37, 2019.

\bibitem[Dou22]{D22}
S.~Douba.
\newblock 2-systems of arcs on spheres with prescribed endpoints.
\newblock {\em Michigan Mathematical Journal}, 71:321--346, 2022.

\bibitem[EG73]{EG73}
P.~Erd{\H{o}}s and R.~K. Guy.
\newblock Crossing number problems.
\newblock {\em American Mathematical Monthly}, 80:52--58, 1973.

\bibitem[Ele97]{E97}
G.~Elekes.
\newblock On the number of sums and products.
\newblock {\em Acta Arithmetica}, 81(4):365--367, 1997.

\bibitem[ERS66]{ERS66}
P.~Erd{\H{o}}s, A.~R{\'e}nyi, and V.~T. S{\'o}s.
\newblock On a problem of graph theory.
\newblock {\em Studia Scientiarum Mathematicarum Hungarica}, 1:215--235, 1966.

\bibitem[FK19]{FK19}
R.~Fulek and J.~Kyn{\v c}l.
\newblock Counterexample to an extension of the hanani-tutte theorem on the
  surface of genus 4.
\newblock {\em Combinatorica}, 39(6):1267--1279, 2019.

\bibitem[FPS20]{FPS20}
R.~Fulek, M.~J. Pelsmajer, and M.~Schaefer.
\newblock Strong hanani-tutte for the torus.
\newblock {\em arXiv preprint}, 2020.

\bibitem[FPS21]{FPS21}
J.~Fox, J.~Pach, and A.~Suk.
\newblock On the number of edges of separated multigraphs.
\newblock In {\em Graph Drawing and Network Visualization}, pages 223--227,
  Cham, 2021. Springer.

\bibitem[FPS{\v S}11]{FPSS11}
R.~Fulek, M.~Pelsmajer, M.~Schaefer, and D.~{\v S}tefankovi{\v c}.
\newblock Adjacent crossings do matter.
\newblock In {\em Graph Drawing}, volume 7034 of {\em Lecture Notes in Computer
  Science}, pages 343--354, Berlin, Heidelberg, 2011. Springer.

\bibitem[GISW24]{GISW24}
A.~Gir{\~a}o, F.~Illingworth, A.~Scott, and D.~R. Wood.
\newblock Non-homotopic drawings of multigraphs.
\newblock {\em arXiv preprint}, 2024.

\bibitem[GJ83]{GJ83}
M.~R. Garey and D.~S. Johnson.
\newblock Crossing number is np-complete.
\newblock {\em SIAM Journal on Algebraic Discrete Methods}, 4:312--316, 1983.

\bibitem[Gre19]{G19}
J.~E. Greene.
\newblock On loops intersecting at most once.
\newblock {\em Geometric and Functional Analysis}, 29:1828--1843, 2019.

\bibitem[Guy69]{G69}
R.~K. Guy.
\newblock The decline and fall of zarankiewicz's theorem.
\newblock In {\em Proof techniques in graph theory}, pages 63--69. Academic
  Press, New York, 1969.

\bibitem[HH62]{HH62}
F.~Harary and A.~Hill.
\newblock On the number of crossings in a complete graph.
\newblock {\em Proceedings of the Edinburgh Mathematical Society},
  13(2):333--338, 1962.

\bibitem[HP24]{HP24}
A.~Hubard and H.~Parlier.
\newblock Crossing lemmas for $k$-systems of arcs.
\newblock {\em arXiv preprint}, 2024.

\bibitem[HP25]{HP25}
A.~Hubard and H.~Parlier.
\newblock Crossing number inequalities for curves on surfaces.
\newblock {\em arXiv preprint}, 2025.

\bibitem[HT20]{HT20}
S.-H. Hong and T.~Tokuyama.
\newblock {\em Beyond planar graphs}, volume~1.
\newblock Communications of NII Shonan Meetings, 2020.

\bibitem[Jen71]{J71}
H.~F. Jensen.
\newblock An upper bound for the rectilinear crossing number of the complete
  graph.
\newblock {\em Journal of Combinatorial Theory, Series B}, 10:212--216, 1971.

\bibitem[JMM96]{JMM96}
M.~Juvan, A.~Malni{\v c}, and B.~Mohar.
\newblock Systems of curves on surfaces.
\newblock {\em Journal of Combinatorial Theory, Series B}, 68(1):7--22, 1996.

\bibitem[Kau16]{K16}
M.~Kaufmann.
\newblock Personal communication at the workshop "beyond-planar graphs:
  Algorithmics and combinatorics", 2016.
\newblock Schloss Dagstuhl, Germany, November 6--11.

\bibitem[KPTU18]{KPTU21}
M.~Kaufmann, J.~Pach, G.~T{\'o}th, and T.~Ueckerdt.
\newblock The number of crossings in multigraphs with no empty lens.
\newblock In {\em Graph Drawing}, volume 11282 of {\em Lecture Notes in
  Computer Science}, pages 242--254, Cham, 2018. Springer.

\bibitem[KT23]{KT23}
J.~Karl and G.~T{\'o}th.
\newblock Crossing lemma for the odd-crossing number.
\newblock {\em Computational Geometry}, 108:101901, 2023.

\bibitem[Lei83]{L83}
F.~T. Leighton.
\newblock {\em Complexity Issues in VLSI}.
\newblock MIT Press, Cambridge, 1983.

\bibitem[LS77]{L77}
R.~C. Lyndon and P.~E. Schupp.
\newblock {\em Combinatorial Group Theory}, volume~89 of {\em Ergebnisse der
  Mathematik und ihrer Grenzgebiete}.
\newblock Springer-Verlag, Berlin-New York, 1977.

\bibitem[LT79]{LT79}
R.~J. Lipton and R.~E. Tarjan.
\newblock A planar separator theorem.
\newblock {\em SIAM Journal on Applied Mathematics}, 36:177--189, 1979.

\bibitem[Mat02]{M02}
J.~Matou{\v s}ek.
\newblock {\em Lectures on Discrete Geometry}, volume 212 of {\em Graduate
  Texts in Mathematics}.
\newblock Springer, New York, 2002.

\bibitem[Mir08]{M08}
M.~Mirzakhani.
\newblock Growth of the number of simple closed geodesies on hyperbolic
  surfaces.
\newblock {\em Annals of Mathematics}, 168(1):97--125, 2008.

\bibitem[MRT14]{MRT14}
J.~Malestein, I.~Rivin, and L.~Theran.
\newblock Topological designs.
\newblock {\em Geometriae Dedicata}, 168:221--233, 2014.

\bibitem[PRTT06]{PRTT06}
J.~Pach, R.~Radoicic, G.~Tardos, and G.~T{\'o}th.
\newblock Improving the crossing lemma by finding more crossings in sparse
  graphs.
\newblock {\em Discrete and Computational Geometry}, 36(4):527--552, 2006.

\bibitem[Prz15]{P15}
P.~Przytycki.
\newblock Arcs intersecting at most once.
\newblock {\em Geometric and Functional Analysis}, 25:658--670, 2015.

\bibitem[PS19]{PS19}
P.~Przytycki and C.~Smith.
\newblock Arcs on spheres intersecting at most twice.
\newblock {\em Indiana University Mathematics Journal}, 68:157--178, 2019.

\bibitem[PSS96]{PSS96}
J.~Pach, F.~Shahrokhi, and M.~Szegedy.
\newblock Applications of the crossing number.
\newblock {\em Algorithmica}, 16(1):111--117, 1996.

\bibitem[PS{\v S}07]{PSS07}
M.~Pelsmajer, M.~Schaefer, and D.~{\v S}tefankovi{\v c}.
\newblock Removing even crossings.
\newblock {\em Journal of Combinatorial Theory, Series B}, 97:489--500, 2007.

\bibitem[PS{\v S}09a]{PSS09}
M.~Pelsmajer, M.~Schaefer, and D.~{\v S}tefankovi{\v c}.
\newblock Odd crossing number and crossing number are not the same.
\newblock {\em Discrete and Computational Geometry}, pages 1--13, 2009.

\bibitem[PSS09b]{PSS09a}
M.~J. Pelsmajer, M.~Schaefer, and D.~Stasi.
\newblock Strong hanani--tutte on the projective plane.
\newblock {\em SIAM Journal on Discrete Mathematics}, 23:1317--1323, 2009.

\bibitem[PST99]{PST00}
J.~Pach, J.~Spencer, and G.~T{\'o}th.
\newblock New bounds on crossing numbers.
\newblock In {\em Proceedings of the fifteenth annual symposium on
  Computational Geometry}, pages 124--133, 1999.

\bibitem[PT97]{PT97}
J.~Pach and G.~T{\'o}th.
\newblock Graphs drawn with few crossings per edge.
\newblock {\em Combinatorica}, 17(3):427--439, 1997.

\bibitem[PT00a]{PT00a}
J.~Pach and G.~T{\'o}th.
\newblock Thirteen problems on crossing numbers.
\newblock {\em Geombinatorics}, 9(4):194--207, 2000.

\bibitem[PT00b]{PT00}
J.~Pach and G.~T{\'o}th.
\newblock Which crossing number is it anyway?
\newblock {\em Journal of Combinatorial Theory, Series B}, 80(2):225--246,
  2000.

\bibitem[PT20]{PT20}
J.~Pach and G.~T{\'o}th.
\newblock A crossing lemma for multigraphs.
\newblock {\em Discrete and Computational Geometry}, 63:918--933, 2020.

\bibitem[PTT22]{PTT22}
J.~Pach, G.~Tardos, and G.~T{\'o}th.
\newblock Crossings between non-homotopic edges.
\newblock {\em Journal of Combinatorial Theory, Series B}, 156:389--404, 2022.

\bibitem[Sch12]{S12}
M.~Schaefer.
\newblock The graph crossing number and its variants: A survey.
\newblock {\em Electronic Journal of Combinatorics}, 2012.
\newblock DS21-May.

\bibitem[Sch18]{S18}
M.~Schaefer.
\newblock {\em Crossing numbers of graphs}.
\newblock CRC Press, 2018.

\bibitem[S{\v S}10]{SS10}
M.~Schaefer and D.~{\v S}tefankovi{\v c}.
\newblock Removing independently even crossings.
\newblock {\em SIAM Journal on Discrete Mathematics}, 24:379--393, 2010.

\bibitem[SSSV97]{SSSV97}
F.~Shahrokhi, O.~S{\'y}kora, L.~A. Sz{\'e}kely, and I.~Vr{\v t}o.
\newblock Crossing numbers: bounds and applications.
\newblock In {\em Intuitive Geometry}, volume~6 of {\em Bolyai Society
  Mathematical Studies}, pages 179--206. J{\'a}nos Bolyai Mathematical Society,
  Budapest, 1997.

\bibitem[ST01]{ST01}
J.~Solymosi and C.~D. T{\'o}th.
\newblock Distinct distances in the plane.
\newblock {\em Discrete and Computational Geometry}, 25:629--634, 2001.

\bibitem[Sz{\'e}97]{S97}
L.~Sz{\'e}kely.
\newblock Crossing numbers and hard erd{\H{o}}s problems in discrete geometry.
\newblock {\em Combinatorics, Probability and Computing}, 6(3):353--358, 1997.

\bibitem[T{\'o}t08]{T08}
G.~T{\'o}th.
\newblock Note on the pair-crossing number and the odd-crossing number.
\newblock {\em Discrete and Computational Geometry}, 39(4):791--799, 2008.

\bibitem[Tur41]{T41}
P.~Tur{\'a}n.
\newblock On an extremal problem in graph theory.
\newblock {\em Matematikai {\'e}s Fizikai Lapok}, 48:436--452, 1941.

\bibitem[Tut70]{T70}
W.~T. Tutte.
\newblock Toward a theory of crossing numbers.
\newblock {\em Journal of Combinatorial Theory}, 8:45--53, 1970.

\bibitem[WB83]{WB78}
A.~T. White and L.~W. Beineke.
\newblock Topological graph theory.
\newblock In {\em Selected Topics in Graph Theory}, pages 15--49. Academic
  Press, London-New York, 1983.

\end{thebibliography}

\end{document}